\documentclass[10pt,a4paper,reqno,oneside]{amsart}
\usepackage{amssymb}
\usepackage[english]{babel}
\usepackage[T1]{fontenc}
\usepackage{xcolor}

\usepackage{multirow}

\usepackage[colorlinks]{hyperref}
\usepackage{url}
\usepackage{hyperref}
\definecolor{ForestGreen}{rgb}{0.1,0.6,0.05}
\definecolor{EgyptBlue}{rgb}{0.063,0.1,0.6}
\definecolor{RipeOlive}{HTML}{556B2F}
\hypersetup{
	colorlinks=true,
	linkcolor=EgyptBlue,         
	citecolor=ForestGreen,
	urlcolor=olive
}
\usepackage[hyperpageref]{backref}

\usepackage{enumitem}

\DeclareMathOperator*{\essinf}{ess\,inf}

\usepackage[left=3cm,right=3cm,
    top=3cm,bottom=3cm,bindingoffset=0cm]{geometry}

\newtheorem{theorem}{Theorem}

\newtheorem{lemma}[theorem]{Lemma}
\newtheorem{corollary}[theorem]{Corollary}
\theoremstyle{definition}
\newtheorem{remark}[theorem]{Remark}

\numberwithin{equation}{section}
\numberwithin{theorem}{section}

\begin{document}

\title[On maximum and comparison principles]{On maximum and comparison principles for parabolic problems with the $p$-Laplacian}

\author{Vladimir Bobkov}
\address{Department of Mathematics and NTIS, Faculty of Applied Sciences, University of West Bohemia, Univerzitn\'i 8, 306 14 Plze\v{n}, Czech Republic}
\email{bobkov@kma.zcu.cz}
\thanks{}

\author{Peter Tak\'a\v{c}}
\address{Institut f\"ur Mathematik, Universit\"at Rostock, Universit\"atsplatz 1, D-18055 Rostock, Germany}
\email{peter.takac@uni-rostock.de}
\thanks{}

\begin{abstract}
We investigate strong and weak versions of maximum and comparison principles for a class of quasilinear parabolic equations with the $p$-Laplacian
$$
\partial_t u - \Delta_p u = \lambda |u|^{p-2} u + f(x,t)
$$
under zero boundary and nonnegative initial conditions on a bounded cylindrical domain $\Omega \times (0, T)$, $\lambda \in \mathbb{R}$, and $f \in L^\infty(\Omega \times (0, T))$. 
Several related counterexamples are given.
\end{abstract}

\maketitle

\section{Introduction}
Let $\Omega_T \stackrel{\textrm{def}}{=} \Omega \times (0,T)$ be a parabolic cylinder, where $\Omega \subset \mathbb{R}^N$ ($N \geq 1$) is a bounded domain with Lipschitz boundary $\partial \Omega$, $T \in (0, +\infty)$, and let  $\partial \Omega_T \stackrel{\textrm{def}}{=} \partial \Omega \times (0, T)$ be the corresponding lateral surface.
We consider the following initial-boundary value problem:
\begin{equation}
\label{eq:P}
\tag{$\mathcal{P}$}
\left\{
\begin{aligned}
\partial_t u - \Delta_p u &\,= \lambda |u|^{p-2} u  + f(x,t) , & &(x,t) \in \Omega_T, \\[0.4em]
u(x, 0) &\,= u_0(x), & &x \in \Omega,\\[0.4em]
u(x,t) &\,= 0, & &(x,t) \in \partial \Omega_T.
\end{aligned}
\quad
\right.
\end{equation}
Here,
\begin{math}
\Delta_p u\stackrel{\textrm{def}}{=}
\mbox{div}(|\nabla_x u|^{p-2} \nabla_x u)
\end{math}
is the $p$-Laplacian with the spatial gradient $\nabla_x u$, $p > 1$, and 
$\lambda \in \mathbb{R}$. 
Dealing with \eqref{eq:P}, we assume that the source function
$f\in L^\infty(\Omega_T)$ and initial function
$u_0 \in W_0^{1,p}(\Omega) \cap L^2(\Omega)$, where
$W_0^{1,p}(\Omega)$ is the standard Sobolev space.
Below, by $\lambda_1$ we denote the first positive eigenvalue of
the $p$-Laplacian in $\Omega$ under
the zero Dirichlet boundary conditions, i.e.,
$$
\lambda_1 
=
\inf\left\{
\int_{\Omega} |\nabla u|^p \, dx:~ u \in W_0^{1,p}(\Omega) \text{ with } \int_\Omega |u|^p \, dx = 1
\right\}.
$$

In this article, we study qualitative properties of weak solutions to problem \eqref{eq:P}, such as maximum and comparison principles. 
It is well-known (see, e.g., \cite{ProtterWeinberger} or \cite{smaller}) that in the linear case $p=2$ any (classical) solution $u$ of \eqref{eq:P} satisfies the Weak Maximum Principle (WMP for short), that is, the assumptions $u_0 \geq 0$ in $\Omega$ and $f \geq 0$ in $\Omega_T$ imply that $u \geq 0$ in $\Omega_T$. 
Moreover, the additional assumption $u(x_0, t_0) = 0$ for some $(x_0, t_0) \in \Omega_T$ yields $u \equiv 0$ in $\Omega_{t_0} \stackrel{\textrm{def}}{=} \Omega \times (0,t_0)$, i.e., the Strong Maximum Principle (SMP) holds. 
At the same time, analogous principles for $p \neq 2$ cannot be satisfied, in general, without additional assumptions on the parameter $\lambda$, initial and source data; they are significantly different for the \textit{fast diffusion} (singular case, $p < 2$) and \textit{slow diffusion} (degenerate case, $p>2$).

Consider, for instance, the following shifted Barenblatt solution of the equation $\partial_t u - \Delta_p u = 0$  for $p>2$ \cite{Bar1,kamvaz}:
\begin{equation}\label{bar}
u(x,t) = \frac{1}{(t+\alpha)^k} 
\left(
C - \frac{p-2}{p} \left(\frac{k}{N}\right)^{\frac{1}{p-1}}
\left(\frac{|x|}{(t+\alpha)^{k/N}}\right)^{\frac{p}{p-1}}
\right)_+^{\frac{p-1}{p-2}},
\end{equation}
where $\alpha > 0$, $C>0$, and $k = (p-2+p/N)^{-1}$. It is not hard to see that \eqref{bar} satisfies zero boundary and nonnegative initial conditions on some $\Omega_T$; however, it exhibits a finite speed of propagation, and, in consequence, the SMP does not hold.

On the other hand, for $p<2$, problem \eqref{eq:P} possesses a finite time extinction phenomenon (also known as complete quenching \cite{dao}), that is, there are nonnegative solutions of \eqref{eq:P} which vanish over $\Omega \times \{t_0\}$ for some $t_0 \in (0, T)$ and are strictly positive in $\Omega_{t_0}$. An explicit example of such solution to $\partial_t u - \Delta_p u = 0$ is given, e.g., in \cite[pp.\ 64-65]{dibenedetto2011} for $N=1$ as follows:
\begin{equation}\label{finext}
u(x,t) = \left(
t_0 - (2-p) t
\right)_+^{\frac{1}{2-p}}v(x),
\end{equation}
where $v$ satisfies the equation $-(|v'|^{p-2}v')' = v$ under boundary conditions $v(-1) = v(1) = 0$.

In \cite{BobkovTakac}, we  investigated the SMP for problem \eqref{eq:P} in the case $\lambda \leq 0$ and found suitable versions thereof for $p<2$ and $p>2$ (see also Theorems \ref{SMP} and \ref{SMP2} below, and \cite{nazaret}).
Moreover, we gave several counterexamples alternative to \eqref{bar} and \eqref{finext}.
However, the case $\lambda > 0$ has not been treated in details in \cite{BobkovTakac}.

On the other hand, consider two problems of the type \eqref{eq:P} with ordered initial data $u_0 \leq v_0$ in $\Omega$ and ordered source functions $f \leq g$ in $\Omega_T$. If the corresponding weak solutions $u$ and $v$ are also ordered, i.e., $u \leq v$ in $\Omega_T$, we say that the Weak Comparison Principle (WCP) is satisfied.
Moreover, if the additional assumption $u(x_0, t_0) = v(x_0, t_0)$ for some $(x_0,t_0) \in \Omega_T$ yields $u \equiv v$ in $\Omega_{t_0}$, then 
the Strong Comparison Principle (SCP) holds.
For the linear case $p=2$ the WCP and SCP readily follow from the WMP and SMP, respectively, by considering the difference $v-u$. However, the $p$-Laplacian being nonlinear, it does not allow to use the same method for the general case $p \neq 2$. Hence different arguments have to be employed.
Furthermore, as in the case of maximum principles, the WCP and SCP cannot be satisfied in the general forms for $p \neq 2$, and their appropriate versions crucially depend on $\lambda$, data, and the choice of $p<2$ or $p>2$.

In the present article we concentrate on the SMP for the case $\lambda > 0$ and the WCP and SCP for $\lambda \geq 0$. 
Validity of the Hopf maximum principle (boundary point lemma) is also discussed.
Precise results are formulated in the next section. 
We remark that the right-hand side of \eqref{eq:P} is a model case of more general nonlinearities (cf.\ \cite[Chapter 2]{pucciserrin2007}) and the results we have obtained are typical for general settings.

In Table \ref{tab1} we collect some known information on availability of maximum and comparison principles for \eqref{eq:P} and indicate several open problems. 
The facts without citations are proved in the present article. 
Note that the most of the information on maximum and comparison principles together with corresponding counterexamples is known for the case $\lambda = 0$. 
The WCP (and consequently WMP) in the case $p>1$ and $\lambda \leq 0$ follows, in principle, from the monotonicity of the operator $-\Delta_p$ and term $-\lambda|u|^{p-2}u$, and we refer here, e.g., to \cite[Lemma 3.1]{kilplind} for the case $\lambda = 0$, and to \cite[Lemma 4.9]{takac2010} for $\lambda \leq 0$.
Counterexamples to the SCP in the case $\lambda = 0$ follow from the inspection of solutions of the forms \eqref{bar} and \eqref{finext}. 

Finally, let us remark that, among other qualitative properties of solutions for problem \eqref{eq:P}, wide literature is devoted to Harnack-type inequalities, see \cite{dibenedetto1993, dibenedetto2011}. A version of the antimaximum principle for \eqref{eq:P} has been found in \cite{takac2010}.

\begin{table}[h!]
	\label{tab1}
	\renewcommand{\arraystretch}{1.6} %% increase table row spacing
	\renewcommand{\tabcolsep}{0.2cm}   %% increase table column spacing
	\begin{center}
		\begin{tabular}{|c|c||c|c||c|c|}
			\hline
			\multicolumn{2}{|c||}{}		& WMP & SMP & WCP & SCP 		\\
			\hline
			\hline
			\multirow{2}{*}{$\lambda \leq 0$} & $p<2$	& $+$   \cite{kilplind} & $-$ \cite{dibenedetto1993} / $\pm$  \cite{nazaret} & $+$  \cite{kilplind} & $-$  \cite{dibenedetto1993} / $\pm$ / ? 	\\
			\cline{2-6}
			& $p>2$	& $+$  \cite{kilplind} & $-$ \cite{Bar1} / $\pm$  \cite{BobkovTakac} & $+$  \cite{kilplind} & $-$  \cite{Bar1} / $\pm$	\\
			\hline
			\hline
			\multirow{2}{*}{$\lambda \in (0, \lambda_1]$} & $p<2$	& $+$  \cite{BobkovTakac} & $-$ \cite{BobkovTakac} /  $\pm$ & $-$ / ? & $-$ / ?	\\
			\cline{2-6}
			& $p>2$	& $+$  \cite{BobkovTakac} & $-$ \cite{BobkovTakac} / $\pm$ & $+$ & $-$  \cite{BobkovTakac} / $\pm$	\\
			\hline
			\hline
			\multirow{2}{*}{$\lambda > \lambda_1$} & $p<2$	& $-$ & $-$ & $-$ & $-$	\\
			\cline{2-6}
			& $p>2$ & $+$ & $-$ \cite{BobkovTakac} / $\pm$ & $+$ & $-$  \cite{BobkovTakac}  / $\pm$	\\
			\hline
			\hline
			$\lambda \in \mathbb{R}$ & $p>1$	& $+$  \cite{ProtterWeinberger} & $+$  \cite{ProtterWeinberger} & $+$  \cite{ProtterWeinberger} & $+$  \cite{ProtterWeinberger}	\\
			\hline
		\end{tabular}
	\end{center}
	\hfil
	\smallskip
	\caption{Status of the maximum and comparison principles for problem \eqref{eq:P}.	
		'$+$' - the principle is satisfied; 
		'$-$' - a counterexample is known;
		'$\pm$' - the principle is satisfied under additional assumptions; 
		'$?$' - no satisfactory information.}
\end{table}

\section{Main results}\label{sec:mainres}
In this section we collect our main results.
We recall that all proofs for the linear case $p=2$ are well-known even under more general assumptions on a domain and parabolic operator, see, e.g.,  \cite{ProtterWeinberger} or \cite{smaller}. We include the case $p=2$ in our formulations for the sake of completeness.
For a basic treatment of the nonlinear case $p \neq 2$, including a brief derivation of problem \eqref{eq:P}, we refer to the classical work by {\sc D\'iaz} and {\sc de Th\'elin} \cite{diazdethelin}.

Let $f \in L^\infty(\Omega_T)$, $u_0 \in W^{1,p}(\Omega) \cap L^2(\Omega)$, and let $h$ be a continuous function on $\overline{\partial \Omega_T} = (\overline{\Omega} \times \{0\}) \cup (\partial \Omega \times [0, T])$.
Under a weak solution of the problem
\begin{equation}
\label{eq:P1}
\left\{
\begin{aligned}
\partial_t u - \Delta_p u &\,= \lambda |u|^{p-2} u  + f(x,t) , & &(x,t) \in \Omega_T, \\[0.4em]
u(x, 0) &\,= u_0(x), & &x \in \Omega,\\[0.4em]
u(x,t) &\,= h(x,t), & &(x,t) \in \partial \Omega_T
\end{aligned}
\quad
\right.
\end{equation} 
we mean a Lebesgue-measurable function $u: \Omega_T \to \mathbb{R}$ satisfying
$$
u \in C\left( [0,T] \to  L^2(\Omega) \right) \cap 
L^p\left( (0,T) \to W^{1,p}(\Omega) \right),
$$
and 
\begin{align*}
\int_\Omega u \varphi \, dx \Big|_{t=0}^{t=\tau} &+ \int_{\Omega_\tau}\left( -u 
\,\partial_t \varphi
+ |\nabla_x u|^{p-2} \left<\nabla_x u, \nabla_x \varphi \right> \right) \, dx  \, dt   \\
&= \lambda \int_{\Omega_\tau} |u|^{p-2} u \varphi \, dx \, dt + \int_{\Omega_\tau} f(x,t) \varphi \, dx \, dt
\end{align*}
for every $\tau \in (0, T]$ and for all test functions
$$
\varphi \in W^{1,2} \left( (0,\tau) \to L^2(\Omega) \right) \cap L^p \left( (0, \tau) \to W^{1,p}_0(\Omega) \right).
$$
The boundary condition $u = h$ on $\partial\Omega_T$ holds in the sense of traces of functions $u(\cdot,t)|_{\partial \Omega}$ in $W^{1,p}(\Omega)$ for a.e.\ $t \in (0, T)$.
As usual, $\left<\cdot, \cdot \right>$ denotes the inner product in $\mathbb{R}^N$.

We start with a variant of the Weak Comparison Principle in a subdomain $E \subseteq \Omega_T$.
Consider the following two problems:
\begin{align}
\label{com1}
\partial_t u - \Delta_p u = \lambda |u|^{p-2} u + f \mbox{ in } \Omega_T,~~ 
&u(x, 0) = u_0  \mbox{ in } \Omega,~~
u = h_1 \mbox{ on } \partial \Omega_T,
\\[0.4em]
\label{com2}
\partial_t v - \Delta_p v = \lambda |v|^{p-2} v + g \mbox{ in } \Omega_T,~~ 
&v(x, 0) = v_0  \,\mbox{ in } \Omega,~~
v = h_2 \mbox{ on } \partial \Omega_T.
\end{align}
We assume $f, g \in L^\infty(\Omega)$, $u_0, v_0 \in W^{1,p}(\Omega) \cap L^2(\Omega)$, and $h_1, h_2$ are continuous on $\overline{\partial \Omega_T}$.
\begin{theorem}[WCP]
	\label{WCP2}
	Let $E \subseteq \Omega_T$ be a subdomain.
	Assume that $f \leq g $ a.e.\ in $E$ and let $u, v$ be weak solutions of problems \eqref{com1}, \eqref{com2}, respectively.
	Finally, assume either of the following two conditions:
	\begin{enumerate}[label={\rm(\roman*)}]
		\item\label{WCP2:i} $p > 1$ and $\lambda \leq 0$;
		\item\label{WCP2:ii} $p > 2$, $\lambda > 0$, and $u,v \in L^\infty(E)$.
	\end{enumerate}
	If $u \leq v$ a.e.\ in $(\overline{\Omega} \times [0,T)) \setminus E$, then $u \leq v$ holds throughout $E$.
\end{theorem}

\begin{remark}
	We point out that the initial and boundary conditions are
	included in our hypothesis
	$u\leq v$ a.e.\ in
	\begin{math}
	E^{\mathrm{c}} = (\overline{\Omega} \times [0,T)) \setminus E ,
	\end{math}
	the complement of
	$E\subset \Omega_T = \Omega\times (0,T)$ in the set
	$\overline{\Omega} \times [0,T)$,
	so that $E^{\mathrm{c}}$ contains both sets,
	$\overline{\Omega} \times \{ 0\}$ and
	$\partial\Omega_T = \partial\Omega\times (0,T)$.
	Consequently, the initial conditions are prescribed on
	$\overline{\Omega} \times \{ 0\}\subset E^{\mathrm{c}}$
	whereas boundary conditions are prescribed on
	$\partial\Omega_T\subset E^{\mathrm{c}}$.
	The role played by the set $E\subset \Omega_T$ is to deal
	with the case of the functions
	$f,g\in L^\infty(\Omega)$ satisfying $f\leq g$ a.e.\
	only in some subdomain $E\subset \Omega_T$, i.e.,
	only {\em locally\/}.
\end{remark}

\begin{remark}\label{rem:wcp}
	Evidently, the WCP implies the WMP by taking $u \equiv 0$ as a solution to \eqref{com1} under the trivial initial and boundary data.
	Moreover, we know that the WMP for \eqref{eq:P} is also satisfied in the case  $p < 2$ and $\lambda \in (0, \lambda_1]$ (see \cite[Theorem 2.4]{BobkovTakac}). However, to the best of our knowledge, availability of the WCP for $p < 2$ and $\lambda \in (0, \lambda_1]$ is still an open problem for \textit{nonnegative} source and initial data.
	At the same time, in Section \ref{sec:nonunieq} below we present a counterexample to the WCP for $p<2$, $\lambda > 0$, and appropriately chosen sign-changing source functions. 
	Moreover, for $p<2$ and $\lambda > \lambda_1$ the WCP is violated even under the trivial source and initial data as it is also shown in Section \ref{sec:nonunieq}.
\end{remark}

Now we state the Strong Maximum Principle for problem \eqref{eq:P}.
%We will say that function $u \in C^{1,0}(\Omega_T)$ whenever $u(\cdot,t)$ is continuously differentiable in $\Omega$ for any $t \in (0,T)$, and $u(x,\cdot)$ is continuous in $(0,T)$ for all $x \in \Omega$.
\begin{theorem}[SMP]
	\label{SMP}
	Assume that $f \in L^\infty(\Omega_T)$, $f \geq 0$ a.e.\ in $\Omega_T$ and $u_0 \in W_0^{1,p}(\Omega) \cap L^2(\Omega)$, $u_0 \geq 0$ a.e.\ in $\Omega$. 
	Let $u \in C^{1,0}(\Omega_T)$ be a weak solution of \eqref{eq:P}. 
	Then the following assertions are valid:
	\begin{enumerate}[label={\rm(\roman*)}]
		\item\label{SMP:i} If $p < 2$, $\lambda \leq \lambda_1$, and $u_0 \not\equiv 0$ in $\Omega$, then there exists $\bar{t} \in (0, T]$ such that $u > 0$ in $\Omega_{\bar{t}}$.
		\item\label{SMP:ii} If $p = 2$, $\lambda \in \mathbb{R}$, and $u_0 \not\equiv 0$ in $\Omega$, then $u > 0$ in $\Omega_T$.
		\item\label{SMP:iii} If $p > 2$, $\lambda \leq 0$, and $\essinf\limits_{\mathcal{K}} u_0 > 0$ for any compact subset $\mathcal{K}$ of $\Omega$, then $u > 0$ in $\Omega_T$.
		\item\label{SMP:iv} If $p > 2$, $\lambda > 0$, and $\essinf\limits_{\mathcal{K}} u_0 > 0$ for any compact subset $\mathcal{K}$ of $\Omega$, assume also $u \in L^\infty(\Omega_T)$. Then $u > 0$ in $\Omega_T$.
	\end{enumerate}
\end{theorem}
\begin{corollary}\label{SMPcor}
	Let $p<2$ and $\lambda \leq \lambda_1$. 
	Then the conclusion of {\rm Theorem \ref{SMP} \ref{SMP:i}} is equivalent to $u > 0$ in $\Omega_{\bar{t}(u)}$, where
	\begin{equation}\label{eq:max}
	\bar{t}(u) \stackrel{\mathrm{def}}{=} \max\{t \in (0, T]:~ u > 0 ~\mathrm{in}~ \Omega_t\} > 0.
	\end{equation}
	Moreover, $\bar{t}(u)$ coincides with the following value:
	%	\begin{equation*}
	%	t^*(u) \stackrel{\mathrm{def}}{=} \max\{t \in (0, T]:~ \mathrm{for\ each\ } s \in (0,t) \mathrm{\ there\ is\ some\ } x_s \in \Omega \mathrm{\ such\ that\ } u(x_s,s) > 0\}.
	%	\end{equation*}
	\begin{equation*}
	t^*(u) \stackrel{\mathrm{def}}{=} \max\{t \in (0, T]:~ \forall s \in (0,t)~  \exists x_s \in \Omega \mathrm{\ such\ that\ } u(x_s,s) > 0\}.
	\end{equation*}
\end{corollary}

Note that counterexamples \eqref{bar}, \eqref{finext}, and \cite[pp.\ 226-227]{BobkovTakac} show that the restriction of $\Omega_T$ to $\Omega_{\bar{t}(u)}$ in assertion \ref{SMP:i} and additional assumption $u_0 > 0$ in assertion \ref{SMP:iii} are essential, and, in general, cannot be removed. 

\begin{remark}\label{rem:reg}
	Assume that $\partial \Omega \in C^{1+\alpha}$, $\alpha \in (0,1)$, and a weak solution $u$ of \eqref{eq:P} satisfies $u \in L^\infty(\Omega_T)$. Then $u \in C^{1+\beta, (1+\beta)/2}(\overline{\Omega}\times [\tau, T])$ for any $\tau \in (0, T)$, where $\beta \in (0,1)$ is independent of $u$; see \cite[Theorem 0.1]{lieberman1993} (or \cite[Lemma 4.6]{takac2010} for notations used in the present article). Moreover, if the initial data $u_0 \in C^{1+\beta}(\overline{\Omega})$, then $u \in C^{1+\beta, (1+\beta)/2}(\overline{\Omega_T})$.
	Here, $C^{1+\beta, (1+\beta)/2}(\overline{\Omega}\times [\tau, T])$ is the standard parabolic H\"older space, see, e.g., \cite[(0.6), p.\ 552]{lieberman1993}.
\end{remark}

\begin{remark}
	Considering $f \equiv 0$ in $\Omega_{t_0}$ for some $t_0 \in (0, T)$, and $u_0 \equiv 0$ in $\Omega$, we see that the strict positivity of $u$ in $\Omega_T$ can be violated by taking $u \equiv 0$ in $\Omega_{t_0}$.
	For the existence of a local in time nontrivial solution of \eqref{eq:P} in $\Omega \times [t_0, t_0 + \varepsilon)$ we refer the reader to  \cite[Appendix A]{takac2010} and references therein.
\end{remark}

Let us give another version of the SMP which does not depend on the assumption $u_0 \not\equiv 0$.
\begin{theorem}\label{SMP2}
	Assume that $f \in L^\infty(\Omega_T)$, $f \geq 0$ a.e.\ in $\Omega_T$ and $u_0 \in W_0^{1,p}(\Omega) \cap L^2(\Omega)$, $u_0 \geq 0$ a.e.\ in $\Omega$. 
	Let $u \in C^{1,0}(\Omega_T)$ be a weak solution of \eqref{eq:P} and there exists $(x_0, t_0) \in \Omega_T$ such that $u(x_0, t_0) = 0$. 
	Then the following assertions are valid:
	\begin{enumerate}[label={\rm(\roman*)}]
		\item\label{SMP2:i} If $p < 2$ and $\lambda \leq \lambda_1$, then $u(x, t_0) = 0$ for all $x \in \Omega$.
		\item\label{SMP2:ii} If $p = 2$ and $\lambda \in \mathbb{R}$, then $u(x, t) = 0$ for all $(x, t) \in \Omega\times(0, t_0]$.
		\item\label{SMP2:iii} If $p>2$ and $\lambda \leq 0$, then $u(x_0, t) = 0$ for all $t \in (0, t_0]$.
		\item\label{SMP2:iv} If $p>2$ and $\lambda > 0$, assume also $u \in L^\infty(\Omega_T)$. Then $u(x_0, t) = 0$ for all $t \in (0, t_0]$.
	\end{enumerate}
\end{theorem}

\begin{remark}
	The conclusions of assertions (i) of Theorems \ref{SMP} and \ref{SMP2} remain valid for the case $p<2$ and $\lambda > \lambda_1$ if we know \textit{a priori} that a considered solution $u$ of \eqref{eq:P} is nonnegative in $\Omega_T$, i.e., it satisfies the WMP. In this case we can apply the SMP to \eqref{eq:P} with the source function $\tilde{f} = \lambda |u|^{p-2} u + f$, $\tilde{f} \geq 0$ in $\Omega_T$.	
	However, the counterexample in Section \ref{sec:nonunieq} indicates that the WMP for $p<2$ and $\lambda > \lambda_1$ may be violated, in general.
\end{remark}

\begin{remark}
	In Theorems \ref{SMP} and \ref{SMP2} the zero boundary condition $u = 0$ on $\partial\Omega_T$ is used, in fact, to treat the case $p<2$ and $\lambda \in (0, \lambda_1]$, only. 
	(In this case we can guarantee that $u \geq 0$ in $\Omega_T$, see \cite[Theorem 2.4]{BobkovTakac}.)
	In all other considered cases (i.e., $p < 2$, $\lambda \leq 0$, and $p > 2$, $\lambda \in \mathbb{R}$), the assertions  of Theorems \ref{SMP} and \ref{SMP2} hold for corresponding solutions of \eqref{eq:P1} with $h \geq 0$ on $\partial \Omega_T$, since the WCP given by Theorem \ref{WCP2} implies the WMP.
	In particular, Theorem \ref{SMP} \ref{SMP:iii} implies that weak $C^{1,0}$-solutions of \eqref{eq:P1} with $p>2$, $\lambda \in \mathbb{R}$, and $h \geq 0$ cannot extinct in a finite time.
\end{remark}

A further important development of \textit{maximum principle} properties of solutions to problem \eqref{eq:P} can be given by the Hopf Maximum Principle (HMP).

\begin{theorem}[HMP]\label{HMP}
	Assume that $f \in L^\infty(\Omega_T)$, $f \geq 0$ a.e.\ in $\Omega_T$ and $u_0 \in W_0^{1,p}(\Omega) \cap L^2(\Omega)$, $u_0 \geq 0$ a.e.\ in $\Omega$. Assume also that $\partial \Omega$ satisfies the interior sphere condition at a point $(x_1, t_1) \in \partial \Omega_T$. Let 
	$u \in C^{1,0}(\Omega_T \cup \{(x_1, t_1)\})$ be a weak solution of \eqref{eq:P}. If $u(\cdot,t_1) > 0$ in $\Omega$ and either 
	\begin{enumerate}[label={\rm(\roman*)}]
		\item\label{HMP:i} $p < 2$ and $\lambda \leq \lambda_1$, or
		\item\label{HMP:ii} $p = 2$ and $\lambda \in \mathbb{R}$, 
	\end{enumerate}
	then the outer normal derivative of $u$ at $(x_1, t_1)$ is strictly negative, i.e.,
	$$
	\frac{\partial u (x_1,t_1)}{\partial \nu} < 0,
	$$
	where $\nu$ is the outer unit normal to $\partial \Omega_T$ at $(x_1,t_1)$.
\end{theorem}

By counterexample \eqref{bar} or \cite[p.\ 229]{BobkovTakac} we know that the HMP is violated for  $p > 2$ and $\lambda \in \mathbb{R}$.

\medskip
Finally we discuss the Strong Comparison Principle. 

\noindent
\textbf{Hypothesis.} 
We assume that $\Omega$ is of class $C^{1+\alpha}$ for some $\alpha \in (0,1)$ and satisfies the interior sphere condition. 

Consider the following initial-boundary value problems of the type \eqref{eq:P}:
\begin{align}
\label{l=01}
\partial_t u - \Delta_p u = \lambda|u|^{p-2} u + f \mbox{ in } \Omega_T,~~ 
&u(x, 0) = u_0  \mbox{ in } \Omega,~~
u = 0 \mbox{ on } \partial \Omega_T,
\\[0.4em]
\label{l=02}
\partial_t v - \Delta_p v = \lambda|v|^{p-2} v + g \mbox{ in } \Omega_T,~~ 
&v(x, 0) = v_0  \,\mbox{ in } \Omega,~~
v = 0 \mbox{ on } \partial \Omega_T.
\end{align}
Here $0 \leq f \leq g$ a.e.\ in $\Omega_T$ and $0 \leq u_0 \leq v_0$ a.e.\ in $\Omega$.
Let $u$ and $v$ be bounded weak solutions of \eqref{l=01} and \eqref{l=02}, respectively.
From Remark \ref{rem:reg} we know that $u, v \in C^{1+\beta, (1+\beta)/2}(\overline{\Omega}\times [\tau, T])$ for any $\tau \in (0, T)$. Assume also that $\bar{t}(v)$, defined by \eqref{eq:max} for the solution $v$, is strictly positive (it can be achieved, e.g., by taking $v_0 \not\equiv 0$ in $\Omega$, see Corollary \ref{SMPcor}), that is,  
$v > 0$ in $\Omega_{\bar{t}(v)}$.

\begin{theorem}[SCP]
	\label{SCP}
	Let $p<2$ and $\lambda = 0$.
	If there exists $\tau > 0$ such that $u < v$ in $\Omega_\tau$, then
	\begin{equation}
	\label{eq:con1}
	0 \leq u < v 
	\text{ in } \Omega_{\bar{t}(v)} 
	~\text{ and }~
	\frac{\partial v}{\partial \nu} < \frac{\partial u}{\partial \nu} \leq 0 
	\text{ on }
	\partial \Omega_{\bar{t}(v)}.
	\end{equation}
\end{theorem}
In words, Theorem \ref{SCP} states that the local in time strict inequality $u < v$ extends until the maximal time of applicability of the SMP for $v$.
Let us state also the SCP under different conditions.
\begin{theorem}
	\label{SCP2}
	Let $p<2$ and $\lambda = 0$.
	Assume that $u_0, v_0 \in C^{1+\beta}(\overline{\Omega})$. If 
	\begin{equation}
	\label{eq:as1}
	0 \leq u_0 < v_0 
	\text{ in } \Omega
	~\text{ and either }~
	\frac{\partial v_0}{\partial \nu} \leq \frac{\partial u_0}{\partial \nu} < 0 
	\text{ or } 
	\frac{\partial v_0}{\partial \nu} < \frac{\partial u_0}{\partial \nu} \leq 0 
	\text{ on }
	\partial \Omega,
	\end{equation}
	then \eqref{eq:con1} holds.
\end{theorem}

Note that the conditions in \eqref{eq:as1} do not directly yield $u < v$ in some $\Omega_\tau$, since we allow normal derivatives of $u_0$ and $v_0$ to be equal on $\partial \Omega$.

\begin{remark}
	As we have already mentioned, the WCP for $p<2$ and $\lambda \in (0, \lambda_1]$ is still unknown for nonnegative source and initial data. 
	However, if such version of the WCP is obtained, then the corresponding SCP under the assumptions of Theorems \ref{SCP} and \ref{SCP2} will be automatically satisfied. Indeed, Theorems \ref{SCP} and \ref{SCP2} can be applied to \eqref{l=01} and \eqref{l=02} with the source functions 
	$$
	\tilde{f} = \lambda |u|^{p-2} u + f 
	~\text{ and }~
	\tilde{g} = \lambda |v|^{p-2} + g,
	$$	
	where $0 \leq f \leq g$ a.e.\ in $\Omega_T$ (and hence $0 \leq \tilde{f} \leq \tilde{g}$ a.e.\ in $\Omega_T$). 
	In particular, if we know \textit{a priori} that $u \leq v$ in $\Omega$ for $p<2$ and $\lambda > 0$, then the assertions of Theorems \ref{SCP} and \ref{SCP2} hold true.
\end{remark}

In \cite[Remark 4.2]{BobkovTakac} it was indicated that, in general, the SCP may be violated for any $\lambda \in \mathbb{R}$ whenever $p>2$. 
However, with a help of the Hopf maximum principle, we have the following version of the SCP even in this case.
\begin{theorem}\label{SCP3}
	Let $p>2$ and $\lambda \geq 0$. Assume that for any $t \in (0, \bar{t}(v))$ it holds
	\begin{equation}\label{eq:SCP3}
	\frac{\partial v}{\partial \nu}(x,t) < 0
	\quad
	\text{for all}
	\quad
	x \in \partial \Omega.
	\end{equation}
	Then the assertions of \textnormal{Theorems \ref{SCP}} and \textnormal{\ref{SCP2}} remain valid. 
\end{theorem}

\medskip
The rest of the article is organized as follows. 
In Section \ref{sec:WCP}, we prove Theorems \ref{WCP2}, \ref{SMP}, \ref{SMP2}, and \ref{HMP}.
In Section \ref{sec:nonunieq}, we give two counterexamples to the maximum and comparison principles in the case $p<2$ and $\lambda > 0$.
Finally, Section \ref{sec:SCP} is devoted to the proofs of Theorems \ref{SCP}, \ref{SCP2}, and \ref{SCP3}.

\section{Weak Comparison and Strong Maximum Principles}\label{sec:WCP}

We start with the 
\textbf{proof of Theorem \ref{WCP2}}.
Consider the function $(u-v)_+ \stackrel{\textrm{def}}{=} \max\limits\{u - v, 0\}$.
Since we do not know a priori that $(u-v)_+$ is an admissible test function for \eqref{com1} and \eqref{com2}, we apply the approach from \cite[Lemma 3.1, Chapter VI]{dibenedetto1993} based on the Steklov averages. 
Define the Steklov averages of a function $w$ by
$$
w_h(x,t) = 
\left\{
\begin{aligned}
&\frac{1}{h} \int_t^{t+h} w(x, \tau) \, d\tau, &&t \in [0, T-h],\\
&0, &&t > T-h,
\end{aligned}
\right.
$$
where $h \in (0, T)$. 
First, arguing as in \cite[Remark 1.1, Chapter II]{dibenedetto1993}, it can be seen that the definition of the weak solution $u$ of \eqref{com1} is equivalent to the following one:
\begin{align}
\notag
\int_{\Omega \times \{t\}} \partial_t (u_h) \, \varphi \, dx &+ \int_{\Omega \times \{t\}}\left< \left(|\nabla_x u|^{p-2} \nabla_x u \right)_h, \nabla_x \varphi \right> dx \\
\label{eq:def:new1}
&= \lambda \int_{\Omega \times \{t\}} \left(|u|^{p-2} u\right)_h \varphi \, dx + \int_{\Omega \times \{t\}} f(x,t)_h \, \varphi \, dx
\end{align}
for all $h \in (0, T)$, $t \in (0, T-h)$, and $\varphi \in W_0^{1,p}(\Omega)$. The initial data is understood in the sense that $u_h(\cdot,0) \to u_0$ in $L^2(\Omega)$.
Analogous definition is also valid for the weak solution $v$ of \eqref{com2}.

Let us note that $((u-v)_h)_+ \in W_0^{1,p}(\Omega)$ for each $t \in [0, T-h)$.
Testing now \eqref{eq:def:new1} and the corresponding equation for $v$ by $((u-v)_h)_+$, and then subtracting them from each other, we get 
\begin{align}
\notag
&\int_{\Omega \times \{t\}} \partial_t((u-v)_h) ((u-v)_h)_+ \, dx \\
\notag
&+ \int_{\Omega \times \{t\}} \left< \left(|\nabla_x u|^{p-2} \nabla_x u -|\nabla_x v|^{p-2} \nabla_x v \right)_h, \nabla_x ((u-v)_h)_+ \right> dx  \\
\notag
&= \lambda \int_{\Omega \times \{t\}} \left(|u|^{p-2} u - |v|^{p-2} v\right)_h ((u-v)_h)_+ \, dx \\
\label{eq:wcp:new1}
&+ \int_{\Omega \times \{t\}} (f(x,t)-g(x,t))_h ((u-v)_h)_+ \, dx.
\end{align}
Let us integrate this equality over $(0,\tau)$, where $\tau \in (0, T-h)$. First, notice that 
$$
\partial_t((u-v)_h) ((u-v)_h)_+ = \frac{1}{2} \frac{\partial}{\partial t} \left[((u-v)_h)_+\right]^2
$$
which yields the integral
\begin{align}
\notag
\int_{\Omega_\tau} &\partial_t((u-v)_h) ((u-v)_h)_+ \, dx \, dt 
\\
\label{eq:wcp:new2}
&= \frac{1}{2} \left[ \int_\Omega \left(((u-v)_h)_+(x,\tau) \right)^2 \,dx  - \int_\Omega \left(((u-v)_h)_+(x,0) \right)^2 \,dx \right].
\end{align}
Moreover, since $u,v \in C([0,T] \to L^2(\Omega))$ and $u \leq v$ a.e.\ in $(\overline{\Omega} \times [0,T)) \setminus E$,  we have 
$$
\int_\Omega \left(((u-v)_h)_+(x,0) \right)^2 \,dx \to 0 
\quad \text{as } h \to 0.
$$
Letting now $h \to 0$ in \eqref{eq:wcp:new1} and \eqref{eq:wcp:new2}, we obtain
\begin{align*}
&\frac{1}{2}\int_\Omega \left((u-v)_+(x,\tau) \right)^2 \,dx 
+ 
\int_{\Omega_\tau} \left< |\nabla_x u|^{p-2} \nabla_x u -|\nabla_x v|^{p-2} \nabla_x v, \nabla_x (u-v)_+ \right> dx \, dt \\
&= 	\lambda \int_{\Omega_\tau} \left(|u|^{p-2} u - |v|^{p-2} v\right) (u-v)_+ \, dx \, dt
+ \int_{\Omega_\tau} (f(x,t)-g(x,t)) (u-v)_+ \, dx \, dt,
\end{align*}
which implies that
\begin{equation}\label{eq:wcp1x}
\frac{1}{2} \int_\Omega \left((u-v)_+(x,\tau) \right)^2 \,dx \leq 
\lambda \int_{\Omega_\tau} \left(|u|^{p-2} u - |v|^{p-2} v\right) (u-v)_+ \, dx \, dt.
\end{equation}

\ref{WCP2:i}
Assume that $p>1$ and $\lambda \leq 0$. Then \eqref{eq:wcp1x} implies that 
$$
\int_{\Omega} \left((u-v)_+(x, \tau)\right)^2 \,dx \leq 0
\quad 
\text{for all }
\tau \in (0, T), 
$$ 
which yields $u \leq v$ in $\Omega_T$, and therefore $u \leq v$ in $E$. 

\ref{WCP2:ii}
Assume that $p > 2$ and $\lambda > 0$. 
To estimate the right-hand side of \eqref{eq:wcp1x} we use the inequality
$$
(|a|^{p-2}a - |b|^{p-2}b)(a-b) \leq c_1 (|a|+|b|)^{p-2}|a-b|^2,
$$
where $c_1 > 0$ does not depend on $a$, $b \in \mathbb{R}$ (see, e.g., \cite[Appendix A, \S A.2]{takac2010}).
Recalling that $u, v \in L^\infty(E)$, we get
$$
\frac{1}{2} \int_{\Omega} \left((u-v)_+(x, \tau)\right)^2 \,dx \leq C_1 \int_0^\tau \int_{\Omega} \left((u -v)_+\right)^2 \, dx \, dt,
$$
where $C_1 = C_1(\lambda, \tau, u, v) \in (0, +\infty)$ is some constant.
Therefore, Gronwall's inequality implies that $\int_{\Omega} \left((u-v)_+(x, \tau)\right)^2 \,dx \leq 0$ for all $\tau \in (0, T)$, and hence $u \leq v$ in $E$.
\qed

\bigskip
Let us turn to the \textbf{proof of Theorem \ref{SMP}}.
Note that assertion \ref{SMP:ii} is the classical linear case, see, e.g., \cite{ProtterWeinberger} and \cite{smaller}. Moreover, assertions \ref{SMP:i} and \ref{SMP:iii} of Theorem \ref{SMP} were proved in \cite[Theorem 1.1]{BobkovTakac} for $\lambda \leq 0$ assuming that $u(x, \cdot)$ is differentiable with respect to $t \in (0, T)$ for all $x \in \Omega$. 
First, we slightly modify the arguments from \cite[Theorem 1.1]{BobkovTakac} to prove assertion \ref{SMP:i} for all $\lambda \leq \lambda_1$ assuming only $u \in C^{1,0}(\Omega_T)$. 

Since $u_0 \geq 0$ and $u = 0$ on $\partial \Omega_T$, we have $u \geq 0$ in $\Omega_T$ by \cite[Theorem 2.4]{BobkovTakac}.
Let $\Sigma$ be any connected component of the nonempty open set $\{(x,t)\in \Omega_T:~ u(x,t) > 0\} = \Omega_T \cap \{u > 0\} \neq \emptyset$.
If $\Sigma = \Omega_T$, the theorem is proved. So let $\Sigma \neq \Omega_T$ which entails $\partial \Sigma \cap \Omega_T \neq \emptyset$. 
Consequently, there exists an open ball $B_\rho(x^*, t^*) \subset \mathbb{R}^N \times \mathbb{R}$ such that 
$(x^*,t^*) \in \partial \Sigma \cap \Omega_T$ and $B_{2\rho}(x^*, t^*) \subset \Omega_T$.
Inspecting the open set $\Sigma \cap B_{\rho}(x^*, t^*)$, we observe that there is an open ball $B_R(x_0, t_0) \subset \Sigma \cap B_{\rho}(x^*, t^*)$ with $\partial B_R(x_0,t_0) \cap \partial \Sigma \neq \emptyset$, i.e., there is a point $(x_1,t_1) \in \partial B_R(x_0,t_0) \cap \partial \Sigma \subset \overline{B_\rho(x^*,t^*)} \subset \Omega_T$.
Taking a ball of smaller radius, if necessary, we may assume that $(x_1, t_1)$ is a unique zero  point of $u$ on $\partial B_R(x_0,t_0)$.

For $r \in (0,R)$ sufficiently small to be specified later, such that $B_r(x_1, t_1) \subset \Omega_T$, we  define the domain $D = B_R(x_0, t_0) \cap B_r(x_1, t_1)$. 
The set $\overline{B_R(x_0,t_0)} \setminus B_r(x_1, t_1)$ being compactly contained inside $\Sigma$, we have $\varepsilon \stackrel{\textrm{def}}{=} \inf\{u(x,t):\,  (x,t) \in \overline{B_R(x_0,t_0)} \setminus B_r(x_1, t_1)\} > 0$.
Now consider the function
\begin{equation}
\label{eq:v}
v(x,t)\stackrel{\textrm{def}}{=} \varepsilon \left( e^{-\alpha d(x,t)^2} - e^{-\alpha R^2} \right),
\end{equation}
where $\alpha > 0$ and $d(x,t) \stackrel{\textrm{def}}{=} \sqrt{|x-x_0|^2 + |t-t_0|}$ (compare with \cite[Chapter 3, Section 3]{ProtterWeinberger}).
It is easy to see that
$$
0< v \leq \varepsilon \mbox{ in } B_R(x_0, t_0),\quad v = 0 \text{ on } \partial B_R(x_0, t_0), \quad v < 0 \text{ in } \mathbb{R}^N \setminus \overline{B_R(x_0, t_0)},
$$
Moreover, by the definition of $\varepsilon$ we see that $v \leq u$ on $\overline{\Omega_T} \setminus D$ for every $\alpha > 0$.
Straightforward calculations (see also \cite[p.\ 225]{BobkovTakac}) yield
\begin{align*}\label{eq:smp2}
g(x,t) &\stackrel{\textrm{def}}{=} \partial_t v - \Delta_p v - \lambda |v|^{p-2} v \equiv 
-\left(\varepsilon e^{-\alpha d(x,t)^2}\right)^{p-1} 
\biggl\{ 
(p-1)(2\alpha |x-x_0|)^p \\
&\cdot 
\biggl[
1 + \lambda \left(1 - e^{-\alpha (R^2 - d(x,t)^2)}\right)^{p-1} \left[(p-1)(2\alpha |x-x_0|)^p\right]^{-1} 
\\
&- (p-2+N) \left[(p-1)(2\alpha)^{p-1} |x-x_0|^{p-2} \right]^{-1}
\biggr]
+ 2 \alpha \left(\varepsilon e^{-\alpha d(x,t)^2}\right)^{2-p} (t - t_0)
\biggr\}
\end{align*}
for $(x,t) \in B_R(x_0, t_0)$.

Suppose that $x_0\neq x_1$. Then we are able to choose the radius $r\in (0,R)$ so small that $R \geq |x-x_0| > r$ holds for all $(x,t) \in \overline{D}$. Therefore, recalling that $1<p<2$, we have
\begin{align*}
&1 + \lambda \left(1 - e^{-\alpha (R^2 - d(x,t)^2)}\right)^{p-1} \left[(p-1)(2\alpha |x-x_0|)^p\right]^{-1} \\
&\,~- (p-2+N) \left[(p-1)(2\alpha)^{p-1} |x-x_0|^{p-2} \right]^{-1}\\
&\geq
1 - |\lambda| \left[(p-1)(2\alpha r)^p\right]^{-1} - (p-2+N) \left[(p-1)(2\alpha)^{p-1} R^{p-2} \right]^{-1} \geq \frac{1}{2}
\end{align*}
for a sufficiently large $\alpha > 0$. Taking $\alpha$ even larger, if necessary, we obtain
\begin{align*}
g(x,t) &\leq
-\left(\varepsilon e^{-\alpha d(x,t)^2}\right)^{p-1} 
\biggl\{ 
\frac{p-1}{2}(2\alpha r)^p \, - 2 \alpha \left(\varepsilon e^{-\alpha d(x,t)^2}\right)^{2-p} |t - t_0|
\biggr\} \leq 0
\end{align*}
for $(x,t) \in D$. Therefore $g \leq 0 \leq f$ in $D$. 

Consider the case $\lambda \leq 0$. Recalling that $v \leq u$ on $\overline{\Omega_T} \setminus D$, we apply Theorem \ref{WCP2}, we get $v \leq u$ in $D$.
On one hand, since $u \geq 0$ in $\Omega_T$,  $u(x_1,t_1) = 0$ at $(x_1,t_1) \in \Omega_T$, and $u \in C^{1}(\Omega\times \{t_1\})$, we see that $\nabla_x u(x_1,t_1) = 0$.
On the other hand, $v > 0$ in $D$, $v(x_1,t_1) = 0$, $v \in C^{1}(\Omega \times \{ t_1\})$, but  $\nabla_x v(x_1,t_1) \neq 0$, since
$$
\left<\nabla_x v(x_1,t_1), (x_1 - x_0)\right>_{\mathbb{R}^N} = 
-2\alpha \varepsilon |x_1-x_0|^2 e^{-\alpha R^2} < 0.
$$
This is a contradiction to $v \leq u$ near $x_1$ in $D$. Thus, we have $x_0 = x_1$, i.e., $|t_0-t_1| = R$. The last fact allows us to argue as in the proof of \cite[Lemma 9.10, Chapter 9, \S{B}, p.\ 86]{smaller} (or, equivalently, $(N+1)$-dimensional generalization of \cite[Lemma 2, Chapter 3, Section 2, p.\ 166]{ProtterWeinberger}) to derive that $u(\cdot,t_1)\equiv 0$ in $\Omega$.

Assume now that $\lambda \in (0, \lambda_1]$. Recall that $u \geq 0$ in $\Omega_T$. Considering the function $\tilde{f} = \lambda |u|^{p-2} u + f$, we get $\tilde{f} \geq 0$ a.e.\ in $\Omega_T$. Hence, applying the proof from above to problem \eqref{eq:P} with the source function $\tilde{f}$, we obtain again that $u(\cdot,t_1)\equiv 0$ in $\Omega$.

Finally, consider
\begin{equation}\label{eq:bart}
\bar{t}(u) \stackrel{\textrm{def}}{=} \inf \{t \in (0,T]:~ \exists \, x\in \Omega \,\text{ such that } u(x,t) = 0 \}
\end{equation}
and put $\bar{t} = T$ whenever $u > 0$ in $\Omega \times (0, T]$.
If $\bar{t}(u) = 0$, then there exists a sequence $\{t_n\}_{n \in \mathbb{N}}$ such that $t_n \to 0$ as $n \to +\infty$ and hence $\|u(\cdot, t_n)\|_{L^2(\Omega)} = 0$ for all $n \in \mathbb{N}$. 
This implies that $\|u_0\|_{L^2(\Omega)} = 0$, since $u \in C\left( [0,T] \to  L^2(\Omega) \right)$. 
However, it contradicts the assumption $u_0\not\equiv 0$ in $\Omega$. Thus, $\bar{t}(u) > 0$ and consequently $u > 0$ in $\Omega_{\bar{t}(u)}$.
It is not hard to see that definitions \eqref{eq:max} and \eqref{eq:bart} coincide.

\medskip
Let us now prove assertions \ref{SMP:iii} and \ref{SMP:iv} of Theorem \ref{SMP}.
From Theorem \ref{WCP2} we know that $u \geq 0$ in $\Omega_T$.
Suppose, by contradiction, that there exists $(x_0,t_0) \in \Omega_T$ with $u(x_0,t_0) = 0$.
To exclude this case, we construct an appropriate nonnegative subsolution to \eqref{eq:P} (different than \eqref{eq:v}) which is strictly positive at $(x_0,t_0)$, and apply the WCP to get a contradiction. 
Assume, without loss of generality, that $x_0 = 0$, and let $K_{R}$ be an open $N$-dimensional ball with radius $R$ centered at the origin, such that $\overline{K_{R}} \subset \Omega$.
Consider the function $w: K_{R} \times [0,T] \to \mathbb{R}$ given by
\begin{equation}\label{eq:subsol}
w(x,t) = C (R^2 - |x|^2)^m (T-t),
\end{equation}
where constants $C > 0$ and $m \geq 2$ will be specified later.
To avoid confusion with notations, let us denote the radial variable $s \stackrel{\textrm{def}}{=} |x|$. Since $w(x,t) = w(s,t)$, we have
$$
\Delta_p w \equiv (p-1) |w'_s|^{p-2} w''_{ss} + (N-1)s^{-1}|w'_s|^{p-2} w'_{s},
$$
and direct calculations imply
\begin{align*}
\partial_t w &- \Delta_p w - \lambda |w|^{p-2}w =
-C (R^2 - s^2)^m  \\
\cdot &\biggl(1- C^{p-2} (2m)^{p-1} s^{p-2} (R^2 - s^2)^{m(p-2)-p} \times \\
&\left(R^2(p+N-2) - ((2m-1)(p-1)+N-1)s^2\right) (T-t)^{p-1} \\
&+\lambda C^{p-2} (R^2 - s^2)^{m(p-2)} (T-t)^{p-1}\biggr).
\end{align*}
Choosing $m \geq \frac{p}{p-2}$ and recalling that $p > 2$, we see that all terms are uniformly bounded in $K_R \times [0,T]$.
Hence, taking $C>0$ small enough, we derive that $\partial_t w - \Delta_p w - \lambda |w|^{p-2}w \leq 0 \leq f$ in $K_R \times [0,T]$.
On one hand, since $w = 0$ on $\partial K_R \times [0,T]$, we have $w \leq u$ on $\partial K_R \times [0,T]$. On the other hand, since $\overline{K_R} \subset \Omega$ and $u_0 > 0$ locally uniformly in $\Omega$, we can find (if necessary) smaller $C>0$ to satisfy $w \leq u$ in $K_R \times \{0\}$. 
Therefore, applying Theorem \ref{WCP2}, we deduce that $w \leq u$ in $K_R \times [0,T]$. However,  $w(x_0,t_0) > 0$, which contradicts the assumption $u(x_0,t_0) = 0$.

\qed

\medskip
\textbf{Corollary \ref{SMPcor}} follows directly from the proof of assertion \ref{SMP:i} of Theorem \ref{SMP}. 

\medskip
Now we discuss the 
\textbf{proof of Theorem \ref{SMP2}}.
Assertion \ref{SMP2:i} of Theorem \ref{SMP2} follows again from the proof of assertion \ref{SMP:i} of Theorem \ref{SMP}. Assertion \ref{SMP2:ii} can be found in \cite{ProtterWeinberger} and \cite{smaller}.

Let us prove assertions \ref{SMP2:iii} and \ref{SMP2:iv}.
Assume that $u(x_0,t_0) = 0$ and, without loss of generality, $x_0 = 0$.
Suppose, by contradiction, that there exists $t_1 \in (0, t_0)$ with $u(x_0, t_1) > 0$. 
Thus, due to the continuity of $u$, there exists $N$-dimensional open ball $K_R$ centered at $x_0 = 0$ such that $\overline{K_R} \subset \{(x,t_1) \in \Omega_T:~ u(x,t_1) > 0\}$.
Considering the function $w: K_R \times [t_1, T] \to \mathbb{R}$ defined by \eqref{eq:subsol}, 
we argue as in the proof of assertions \ref{SMP:iii} and \ref{SMP:iv} of Theorem \ref{SMP} to derive that $w \leq u$ in $K_R \times [t_1, T]$, which leads to a contradiction.
\qed

\medskip
The Hopf Maximum Principle stated in \textbf{Theorem \ref{HMP}} \ref{HMP:i} can be proved using the same arguments as in the proof of assertion \ref{SMP:i} of Theorem \ref{SMP} by considering the subsolution \eqref{eq:v} in ball $B_R(x_0, t_0) \subset \{(x,t)\in \Omega_T:~  u(x,t) > 0\} \neq \emptyset$ which touches $\partial\Omega_T$ at the point $(x_1,t_1)$ (i.e., $t_1 = t_0$). 
Assertion \ref{HMP:ii} of Theorem \ref{HMP} can be found in \cite{ProtterWeinberger} and \cite{smaller}.

\section{Nonuniqueness results}\label{sec:nonunieq}
In this section we give two counterexamples to the maximum and comparison principles for problem \eqref{eq:P}.

First we present a counterexample to the WCP in the case $p < 2$, 
$\lambda > 0$, and $f$ is a specially chosen sign-changing function in $\Omega_T$. 
For this end, we modify Example 2 from 
\cite[p.\ 148]{Fleckinger1997} on nonuniqueness of nontrivial weak solutions for an elliptic problem, in order to produce corresponding nonuniqueness for the parabolic problem \eqref{eq:P}. This will eventually lead to a violation of the WCP.

Let $p < 2$, $\lambda > 0$, and $\partial \Omega$ be of class $C^2$, for simplicity.
Consider the following elliptic problem:
\begin{equation}
\label{eq:D1}
\left\{ 
\begin{aligned}
- \Delta_p w &= \lambda |w|^{p-2} w - w + h(x), &x &\in \Omega, \\[0.4em]
w &= 0, &x &\in \partial\Omega,
\end{aligned} 
\right.
\end{equation}
where $h \in L^\infty(\Omega)$ will be specified later. 
We assume, without loss of generality, that $0 \in \Omega$.
The corresponding energy functional is given by
$$
E_\lambda(w) = \frac{1}{p}  \int_{\Omega} |\nabla w|^p \,dx - 
\frac{\lambda}{p}  \int_{\Omega} |w|^p \,dx + 
\frac{1}{2}  \int_{\Omega} |w|^2 \,dx - 
\int_{\Omega} h(x)\, w \,dx.
$$
Note that problem \eqref{eq:D1} is supercritical whenever $p < \frac{2N}{N+2}$ and $N \geq 2$. However, $E_\lambda$ is well-defined and coercive on the reflexive Banach space $X = W_0^{1,p}(\Omega) \cap L^2(\Omega)$ endowed with the norm $\|\cdot\| = \|\cdot\|_{W_0^{1,p}(\Omega)} + \|\cdot\|_{L^2(\Omega)}$. Clearly, $E_\lambda$ is Fr\'echet differentiable on $X$. 
The coercivity on $X$ is guaranteed jointly by the first summand, $\frac{1}{p}  \int_{\Omega} |\nabla w|^p \,dx$, and the third summand, $\frac{1}{2}  \int_{\Omega} |w|^2 \,dx$. Indeed, the latter dominated both, the second and fourth summands, by H\"older's inequality.

Our aim is to find two different critical points of $E_\lambda$ in $X$. First we construct a critical point of $E_\lambda$ of saddle type for a special choice of $h$.
Fix some $\varepsilon_1 \in (0,1)$ such that the ball $B_{2\varepsilon_1}(0)$ is compactly contained in $\Omega$,
and let a function $w_0$ be defined as
\begin{equation*}
\left\{ 
\begin{aligned}
&w_0(x) = |x|^m 					&&\text{ for }~ |x| \leq \varepsilon_1 \\
&w_0(x) = 0							&&\text{ for }~ x \in \Omega \setminus B_{2\varepsilon_1}(0), \\
&w_0 \in C^2(\overline{\Omega}), 	&&\Delta_p w_0 \in C(\overline{\Omega}),
\end{aligned} 
\right.
\end{equation*}
where $m \geq N (2-p)^{-1}$ (this choice will be clear later). Such $w_0$ can be explicitly constructed inside the annulus $B_{2\varepsilon_1}(0) \setminus B_{\varepsilon_1}(0)$ as a polynomial of the radial variable $|x|$.
Define now $h \in C(\overline{\Omega})$ by
\begin{equation}
\label{g(x)}
h(x) = -\Delta_p w_0(x) - \lambda |w_0(x)|^{p-2} w_0(x) + w_0(x), \quad x \in \Omega.
\end{equation}
By construction, $w_0 \in X$ and satisfies \eqref{eq:D1} pointwise with the source function $h$ given by \eqref{g(x)}, and therefore $w_0$ is a critical point of $E_\lambda$. 
However, we claim that $E_\lambda$ does not attain its local minimum at $w_0$.
For this end, let us take any function $z \in C^2(\overline{\Omega})$ such that
\begin{equation*}
\left\{ \begin{aligned}
z(x) &= 1 					&&\text{ for }~ |x| \leq \varepsilon, \\
z(x) &= z(|x|) \in [0, 1] 	&&\text{ for }~ \varepsilon \leq |x| \leq \varepsilon_1, \\
z(x) &= 0 					&&\text{ for }~ x \in \Omega \setminus B_{\varepsilon_1}(0),
\end{aligned} \right.
\end{equation*}
where $\varepsilon \in (0, \varepsilon_1)$, and show that $E_\lambda(w_0 + tz) < E_\lambda(w_0)$ for $t > 0$ small enough. 
By the mean value theorem, for any $t > 0$ there exists $t_0 \in (0, t)$ such that
\begin{equation}\label{eq:taylor}
E_\lambda(w_0 + tz) =  E_\lambda(w_0) + 
t \left< E'_\lambda(w_0 + t_0 z), z \right>.
\end{equation}
Let us investigate the function $\zeta: (0, +\infty) \to \mathbb{R}$ defined as
$$
\zeta(t) 
=
\frac{1}{t} \left< E'_\lambda(w_0 + t z), z \right> \equiv 
\frac{1}{t} \left( \left< E'_\lambda(w_0 + t z), z \right> - \left< E'_\lambda(w_0), z \right> \right).
$$
It is not hard to see that $\zeta \in C(0,+\infty)$. Moreover, since $w_0$ and $z$ are radial in $B_{\varepsilon_1}(0)$, we have
\begin{align}
\notag
\zeta(t) =
\notag
&= \frac{N \omega_N}{t} \biggl(
\int\limits_{\varepsilon}^{\varepsilon_1}
r^{N-1} \left| w_0' + t z'\right|^{p-2} (w_0' + t z') z' \, dr - 
\int\limits_{\varepsilon}^{\varepsilon_1}
r^{N-1} \left| w'_0 \right|^{p-2} w'_0  z' \, dr
\biggr) \\
\notag
&- \frac{\lambda N \omega_N}{t} \biggl(
\int\limits_0^{\varepsilon_1}
r^{N-1} \left| w_0 + t z \right|^{p-2} (w_0 + t z) z \, dr - 
\int\limits_0^{\varepsilon_1}
r^{N-1}\left| w_0 \right|^{p-2} w_0  z \, dr
\biggr) \\
&+ 
N \omega_N \int\limits_0^{\varepsilon_1}
r^{N-1} z^2 \, dr 
\label{E0}
= N \omega_N \biggl( \widetilde{E}_\lambda(t) + \hat{E}_\lambda(t) + \int\limits_0^{\varepsilon_1} r^{N-1} z^2 \, dr \biggr),
\end{align}
where $\omega_N$ is the volume of a unit ball in $\mathbb{R}^N$,
\begin{align*}
\widetilde{E}_\lambda(t) &= \frac{1}{t} \biggl(
\int\limits_{\varepsilon}^{\varepsilon_1}
r^{N-1} \left| w'_0 + t z' \right|^{p-2} (w'_0 + t z') z' \, dr - 
\int\limits_{\varepsilon}^{\varepsilon_1}
r^{N-1} \left| w'_0 \right|^{p-2} w'_0  z' \, dr
\biggr) \\
&- \frac{\lambda}{t} \biggl(
\int\limits_{\varepsilon}^{\varepsilon_1}
r^{N-1} \left| w_0 + tz \right|^{p-2} (w_0 + tz) z\, dr - 
\int\limits_{\varepsilon}^{\varepsilon_1}
r^{N-1} \left| w_0 \right|^{p-2} w_0 z \, dr
\biggr),
\end{align*}
and
$$
\hat{E}_\lambda(t) =  - \frac{\lambda}{t} \biggl(
\int_0^{\varepsilon}
r^{N-1} \left| w_0 + t \right|^{p-2} (w_0 + t) \, dr - 
\int_0^{\varepsilon}
r^{N-1}\left| w_0 \right|^{p-2} w_0 \, dr
\biggr).
$$
Since $w_0(r) = r^m$ for $r \leq \varepsilon_1$, we get
\begin{equation}
\label{E1}
\lim_{t \to 0} \hat{E}_\lambda(t) = -\lambda (p-1) \int_0^{\varepsilon} r^{N - 1 + m(p-2)} \, dr  = -\infty
\end{equation}
provided $N - 1 + m (p-2) \leq -1$ (that is, $m \geq N (2-p)^{-1}$) and $\lambda > 0$.
On the other hand, 
\begin{equation}
\label{E2}
\widetilde{E}_\lambda(0) = \lim_{t \to 0} \widetilde{E}_\lambda(t) = (p-1) \int_{\varepsilon}^{\varepsilon_1} r^{N-1} \left( |w'_0|^{p-2} |z'|^2 - \lambda |w_0|^{p-2} z^2 \right) \, dr
\end{equation}
is finite.
Combining \eqref{E0}, \eqref{E1}, and \eqref{E2}, we obtain
\begin{align*}
&\lim_{t \to 0} \zeta(t) =N \omega_N \left(\widetilde{E}_\lambda(0) -\lambda (p-1) \int_0^{\varepsilon} r^{N-1+m(p-2)} \, dr + \int_0^{\varepsilon_1} r^{N-1} z^2 \, dr \right)= -\infty.
\end{align*}
Thus, substituting $\zeta(t_0)$ into \eqref{eq:taylor} and recalling that $\zeta(t_0)$ is continuous for $t_0 > 0$, we see that $E_\lambda(w_0 + tz) < E_\lambda(w_0)$ for sufficiently small $t > 0$, i.e., $E_\lambda(w_0)$ decreases in direction $z$, and hence $w_0$ is not the point of a local minimum of $E_\lambda$ in $X$.

On the other hand, since $E_\lambda$ is coercive and weakly lower semicontinuous on $X$, it possesses a global minimizer $w_1 \in X$ which becomes a second nontrivial weak solution of \eqref{eq:D1}. 
Note that $w_1 \in L^\infty(\Omega)$ (see, e.g., \cite[Th\'eor\`eme A.1, p.\ 96]{anane}), and therefore $w_1 \in C^{1+\alpha}(\overline{\Omega})$ for some $\alpha \in (0,1)$ (cf. \cite{lieberman}).

Let us show now that nonuniqueness for elliptic problem \eqref{eq:D1} causes a corresponding nonuniqueness for the following parabolic problem of the type \eqref{eq:P}:
\begin{equation}
\label{eq:Px}
\left\{ 
\begin{aligned}
\partial_t u - \Delta_p u &= \lambda |u|^{p-2} u + h(x) \, |v(t)|^{p-2} v(t), & (x,t) &\in \Omega_T, \\[0.4em]
u_0 (x) &= 0, & x &\in \Omega, \\[0.4em]
u(x,t) &= 0, & (x,t) &\in \partial \Omega_T,
\end{aligned} 
\right.
\end{equation}
where $h$ is the sign-changing source function defined by \eqref{g(x)} and $v(t)$ is a (nontrivial) positive solution of the Cauchy problem
\begin{equation}\label{Cauchy}
\left\{ 
\begin{aligned}
\partial_t v - |v|^{p-2} v &= 0, & t > 0, \\[0.4em]
v(0) &= 0,
\end{aligned} 
\right.
\end{equation}
where $1<p<2$, which is given by
\begin{equation}\label{Cauchy2}
v(t) = (2 - p)^{\frac{1}{2 - p}} \, t^{\frac{1}{2 - p}}.
\end{equation}
If we look for solutions of \eqref{eq:Px} in the form $u(x,t) = w(x) v(t)$, then
$$
u_0(x,t) = w_0(x) \, v(t)  \quad \mbox{and} \quad u_1(x,t) = w_1(x) \, v(t) 
$$
are two different solutions to \eqref{eq:Px} which implies the desired nonuniqueness. Thus, the WCP is violated, since $u_0 = u_1$ on the parabolic boundary of $\Omega_T$.
\qed

\medskip
Note that it is not possible to obtain a similar nonuniqueness result if $h \geq 0$ a.e.\ in $\Omega$ and $\lambda \leq  \lambda_1$.
First, under these assumptions, any weak solution $w$ of \eqref{eq:D1} is nonnegative. Indeed, testing \eqref{eq:D1} by $w_-$, we obtain
\begin{align*}
0 
= \left( \int_\Omega |\nabla w_-|^p \,dx -  
\lambda \int_\Omega |w_-|^p \,dx\right) + 
\int_\Omega |w_-|^2 \,dx 
&+
\int_\Omega h(x) w_- \,dx \geq 
\int_\Omega |w_-|^2 \,dx \geq
0.
\end{align*}
However, this is possible only if $w \geq 0$ a.e.\ in $\Omega$. 
If $h \equiv 0$ a.e.\ in $\Omega$, then it is not hard to observe that \eqref{eq:D1} has a trivial solution only. Assume that $h \not\equiv 0$ a.e.\ in $\Omega$. 
As was noted above, $w \in C^{1+\alpha}(\overline{\Omega})$ for some $\alpha \in (0,1)$, and hence $w > 0$ in $\Omega$, due to \cite{vazquez1984}. 
If we suppose that \eqref{eq:D1} has two (positive) solutions, then the D\'iaz-Sa\'a inequality (see \cite[Lemma 2]{diazsaa}) leads to a contradiction, and hence the desired uniqueness for \eqref{eq:D1} follows. 
Let us remark also that a counterexample similar to \eqref{eq:Px} cannot be applied to the case $p>2$, since the Cauchy problem \eqref{Cauchy} has a trivial solution only.

\medskip
Now we give a counterexample to all maximum and comparison principles in the case $p < 2$ and $\lambda > \lambda_1$ under the trivial source and initial data.
Consider the following particular case of \eqref{eq:P}:
\begin{equation}
\label{eq:P1x}
\left\{ 
\begin{aligned}
\partial_t u - \Delta_p u &= \lambda |u|^{p-2} u, & (x,t) &\in \Omega_T, \\[0.4em]
u_0 (x) &= 0, & x &\in \Omega, \\[0.4em]
u(x,t) &= 0, & (x,t) &\in \partial \Omega_T.
\end{aligned} 
\right.
\end{equation}
Let $v(t)$ be the solution \eqref{Cauchy2} of the Cauchy problem \eqref{Cauchy},
and let $w$ be a positive weak solution to the following logistic problem (that is, \eqref{eq:D1} with $h \equiv 0$ in $\Omega$):
\begin{equation*}\label{eq:D}
\left\{ \begin{aligned}
- \Delta_p w &= \lambda |w|^{p-2} w - w, &x &\in \Omega, \\[0.4em]
w &= 0 &x &\in \partial \Omega,
\end{aligned} 
\right.
\end{equation*}
see, e.g., \cite[Theorem 1.1, (e), (b), p.\ 947]{ilyasov}.
It is not hard to see that $u_\pm(x,t) = \pm w(x) \, v(t)$ is a pair of positive and negative solutions to \eqref{eq:P1x}, and $u \equiv 0$ is a trivial solution of \eqref{eq:P1x}. Thus, for $u_0 \equiv 0$ and $f \equiv 0$, both weak and strong forms of maximum and comparison principles for \eqref{eq:P} are violated.
\qed

\smallskip
The considered counterexamples indicate that the question about the validity of the WCP for $p < 2$, $\lambda >0$, and nontrivial nonnegative data remains open. 

\smallskip
Finally, we refer the interested reader to \cite{arrieta,bensid,dao2,dao3,dao4,deguchi,diaztello,feireisl,gianni} for the existence, uniqueness, and nonuniqueness results to parabolic problems with some other types of nonlinearity.

\section{Strong Comparison Principles}\label{sec:SCP}

In this section we prove the versions of the SCP given by Theorems \ref{SCP}, \ref{SCP2}, and \ref{SCP3}.

Recall that we consider two initial-boundary value problems of the type \eqref{eq:P}:
\begin{align}
\tag{\ref{l=01}}
\partial_t u - \Delta_p u = \lambda |u|^{p-2} u + f \mbox{ in } \Omega_T,~~ 
&u(x, 0) = u_0  \mbox{ in } \Omega,~~
u = 0 \mbox{ on } \partial \Omega_T,
\\[0.4em]
\tag{\ref{l=02}}
\partial_t v - \Delta_p v = \lambda |v|^{p-2} v + g \mbox{ in } \Omega_T,~~ 
&v(x, 0) = v_0  \mbox{ in } \Omega,~~
v = 0 \mbox{ on } \partial \Omega_T.
\end{align}
We assume that
$\Omega \in C^{1+\alpha}$ and satisfies the interior sphere condition, $\lambda \in \mathbb{R}$ is a constant, $0 \leq f \leq g$ a.e.\ in $\Omega_T$, and 
$0 \leq u_0 \leq  v_0$  a.e.\ in $\Omega$.
Let $u$ and $v$ be bounded  weak solutions of \eqref{l=01} and \eqref{l=02}, respectively, and let $v$ be strictly positive in $\Omega_{\bar{t}(v)}$, where $\bar{t}(v) > 0$ is defined by \eqref{eq:max}. From Remark \ref{rem:reg} we know that $u, v \in C^{1+\beta, (1+\beta)/2}(\overline{\Omega}\times [\tau, T])$ for any $\tau \in (0, T)$.

For $\tau, \delta >0$ small enough we consider an open subset $\mathcal{O}_\delta^\tau$ of $\Omega_{\bar{t}(v)}$ given by
\begin{equation}\label{eq:O}
\mathcal{O}_\delta^\tau \stackrel{\textrm{def}}{=} \mathcal{O}_\delta \times  (\tau, \bar{t}(v)- \tau), \quad \text{where} \quad \mathcal{O}_\delta \stackrel{\textrm{def}}{=} \{x \in \Omega:~ \text{dist}(x, \partial \Omega) < \delta \}.
\end{equation}
To prove Theorems \ref{SCP}, \ref{SCP2}, and \ref{SCP3}, we follow the strategy of \cite{cuestatakac1998}: show first that the SCP holds in  $\mathcal{O}_\delta^\tau$, then  extend it to the whole of $\Omega_{\bar{t}(v)}$.

\begin{lemma}
	\label{scpboundary}
	Let $p<2$ and $\lambda = 0$.
	For any $\tau \in (0, \bar{t}(v)/2)$ there exists $\delta > 0$ such that for every connected component $\Sigma$ of $\mathcal{O}_\delta^\tau$, the equality $u(x_0, t_0) = v(x_0, t_0)$ for some $(x_0, t_0) \in \Sigma$ implies $u \equiv v$ in $\Sigma \cap \{ t \leq t_0 \}$.
\end{lemma}
\begin{proof}
	From Theorems \ref{WCP2}, \ref{SMP}, and \ref{HMP} we know that 
	$$
	0 \leq u \leq v 
	\text{ and }
	0 < v
	\text{ in } \Omega_{\bar{t}(v)}, 
	\quad 
	\frac{\partial v}{\partial \nu} \leq \frac{\partial u}{\partial \nu} \leq 0 
	\text{ and }
	\frac{\partial v}{\partial \nu} < 0
	\text{ on }
	\partial \Omega_{\bar{t}(v)}.
	$$
	Consequently, for any $\tau \in (0, \bar{t}(v)/2)$ there exist $\delta > 0$ and some constants $\eta_1, \eta_2 > 0$ such that $|\nabla_x u| \leq \eta_2$ and $\eta_1 \leq |\nabla_x v|  \leq \eta_2$ throughout $\mathcal{O}_\delta^\tau$.
	Thus, considering $w = v - u$, we have $w \geq 0$ in $\Omega \times \{0\}$, $w = 0$ on $\partial \Omega_{\bar{t}(v)}$, $w \geq 0$ in $\Omega_{\bar{t}(v)}$, and $w \in C^{1+\beta, (1+\beta)/2}(\overline{\Omega}\times [\tau, T])$ for any $\tau \in (0, T)$. Furthermore, subtracting \eqref{l=01} from \eqref{l=02}, we obtain that $w$ weakly satisfies the following linear parabolic inequality in $\mathcal{O}_\delta^\tau$:
	\begin{align}
	\label{eq:oper}
	\partial_t w - \text{div} (A(x, t) \nabla w) = g - f \geq 0.
	\end{align}
	Here, the $(N\times N)$-matrix $A(x, t)$ is obtained via the mean value theorem as
	\begin{align*}
	A(x,t) = \int_0^1 |\nabla_x((1-s) u + s v)|^{p-2}
	\mathbb{A}_p(\nabla_x((1-s) u + s v)) \, ds
	\end{align*}
	where 
	$$	
	\mathbb{A}_p(\vec{a}) \stackrel{\textrm{def}}{=} 
	\mathbb{I} + (p-2) \frac{\vec{a} \otimes \vec{a}}{|\vec{a}|^2} 
	\quad \text{for} \quad 
	\vec{a} \in \mathbb{R}^N \setminus \{\vec{0}\}
	$$
	is a symmetric, positive definite $(N\times N)$-matrix with the eigenvalues $1$ and $p-1$, $\mathbb{I}$ is the identity matrix, $\otimes$ denotes the tensor product, and integration in the definition of $A(x,t)$ is taken componentwise.
	
	Let us show that $A(x,t)$ forms a uniformly elliptic operator in $\mathcal{O}_\delta^\tau$. 
	Recall the inequalities 
	\begin{equation}\label{inq}
	\left(\max_{0 \leq s \leq 1} |a + s b|
	\right)^{p-2} \leq \int_0^1 |a + s b|^{p-2} \, ds
	\leq C_p \left(\max_{0 \leq s \leq 1} |a + s b|
	\right)^{p-2}
	\end{equation}
	from \cite[(A.6), p.\ 645]{takac2010}, which hold with some constant $C_p > 0$, for all $a,b \in \mathbb{R}^N$ with $|a|+|b| > 0$. Using \eqref{inq}, we derive 
	\begin{align*}
	\left(\max_{0 \leq s \leq 1} |\nabla_x((1-s) u + s v)|
	\right)^{p-2} &\leq \int_0^1 |\nabla_x((1-s) u + s v)|^{p-2} \, ds
	\\	
	&\leq C_p \left(\max_{0 \leq s \leq 1} |\nabla_x((1-s) u + s v)|
	\right)^{p-2}
	\end{align*}
	for any $(x, t) \in \mathcal{O}_\delta^\tau$, since $|\nabla_x u| + |\nabla_x(v-u)| > 0$ in $\mathcal{O}_\delta^\tau$. 
	Thus, estimating the quadratic form of $A(x,t)$ by that of $\mathbb{A}_p(\vec{a})$ (see also \cite[(A.10), p.\ 646]{takac2010}), we conclude that there exist $C_1, C_2 > 0$ such that 
	\begin{align*}
	C_1 |\xi|^2 \leq \left<A(x,t) \xi, \xi \right> \leq C_2 |\xi|^2
	\end{align*}
	for all $(x, t) \in \mathcal{O}_\delta^\tau$ and $\xi \in \mathbb{R}^N$.
	That is, the differential operator in \eqref{eq:oper} is uniformly parabolic in $\mathcal{O}_\delta^\tau$, since the corresponding elliptic operator is uniformly elliptic in this domain.
	Therefore, $w$ satisfies the SMP in $\Sigma$, as it follows from the combination of Harnack's inequality \cite{moserparabolic} for weak solutions of equation $\partial_t w - \text{div} (A(x, t) \nabla w) = 0$
	with the classical WCP. See also \cite[Corollary 3.5]{nazarural} for an explicit statement of the SMP for \eqref{eq:oper}. 
	Hence, if $w(x_0, t_0) = 0$ for some $(x_0, t_0) \in \Sigma$, then $w \equiv 0$ in $\Sigma \cap \{ t \leq t_0 \}$, which implies the desired result.
\end{proof}

\medskip
\noindent
\textbf{Proof of Theorem \ref{SCP}}. 
Assume that there exists $\tau > 0$ such that $u < v$ in $\Omega_\tau$. Taking any $\eta \in (0, \tau)$, Lemma \ref{scpboundary} guarantees that $u < v$ in every connected component $\Sigma$ of $\mathcal{O}_\delta^\eta \subset \Omega_{\bar{t}(v)}$ with some $\delta >0$.
Then for any $\delta_1 \in (0, \delta)$ there exists $\alpha > 0$ such that $u + \alpha \leq v$ on $\partial(\Omega \setminus \mathcal{O}_{\delta_1})\times [\eta, \bar{t}(v)-\eta]$.
On the other hand,
\begin{equation*}\label{tech}
\partial_t (u + \alpha) - \Delta_p (u + \alpha) = \partial_t u - \Delta_p u 
\leq 
\partial_t v - \Delta_p v.
\end{equation*}
Hence, Theorem \ref{WCP2} implies that $u+\alpha \leq v$ in $(\Omega \setminus \mathcal{O}_{\delta_1})\times [\eta, \bar{t}(v)-\eta]$. Consequently, $u < v$ in $\Omega \times (0, \bar{t}(v)-\eta]$.
Letting $\eta \to 0$, we conclude that $u < v$ in $\Omega_{\bar{t}(v)}$.
\qed

\medskip
By the same arguments as above we are able to prove Theorem \ref{SCP2}.
Consider the set
$$
\widetilde{\mathcal{O}}_\delta^\tau \stackrel{\textrm{def}}{=} \mathcal{O}_\delta \times  (0, \bar{t}(v) - \tau),
$$
where $\tau, \delta >0$ are sufficiently small and $\mathcal{O}_\delta$ is defined in \eqref{eq:O}.
First, we need the following local SCP.
\begin{lemma}
	\label{scpboundary2}
	Let $p<2$ and $\lambda = 0$.
	Assume, in addition to the assumptions on problems \eqref{l=01}, \eqref{l=02}, that $u_0, v_0 \in C^{1+\beta}(\overline{\Omega})$ satisfy also
	$$
	0 \leq u_0 \leq v_0
	\text{ in } \Omega, 
	\quad 
	\frac{\partial v_0}{\partial \nu} \leq \frac{\partial u_0}{\partial \nu} \leq 0 
	\text{ and }
	\frac{\partial v_0}{\partial \nu} < 0
	\text{ on }
	\partial \Omega.
	$$
	Then for any $\tau \in (0, \bar{t}(v))$ there exists $\delta > 0$ such that for every connected component $\Sigma$ of $\widetilde{\mathcal{O}}_\delta^\tau$, the equality $u(x_0, t_0) = v(x_0, t_0)$ for some $(x_0, t_0) \in \Sigma$ implies $u \equiv v$ in $\Sigma \cap \{ t \leq t_0 \}$.
\end{lemma}
\begin{proof}
	Note that, under the assumptions of the lemma, $u$ and  $v$ belongs to $C^{1+\beta, (1+\beta)/2}(\overline{\Omega}\times [0,T])$, see Remark \ref{rem:reg}.
	Therefore, $v$ satisfies the Hopf maximum principle uniformly on the time interval $[0, \bar{t}(v)-\tau)$ for any fixed $\tau > 0$ which allows us to linearize the $p$-Laplacian in $\widetilde{\mathcal{O}}_\delta^\tau$ as in the proof of Lemma \ref{scpboundary}, and obtain the desired result.
\end{proof}

\smallskip
\noindent
\textbf{Proof of Theorem \ref{SCP2}}.
Under the assumption \eqref{eq:as1}, 
Lemma \ref{scpboundary2} implies that $u < v$ in every connected component $\Sigma$ of $\widetilde{\mathcal{O}}_\delta^\tau$ where $\tau \in (0, \bar{t}(v))$ and $\delta >0$ is small enough.
Hence, the continuity of $u$ and $v$ implies that we can find sufficiently small $\eta>0$ and $\delta_1 \in (0, \delta)$ such that $u < v$ in $(\Omega \setminus \mathcal{O}_{\delta_1}) \times [0, \eta]$.
Thus, $u < v$ in $\Omega_\eta$, and we apply Theorem \ref{SCP} to conclude that $u < v$ in $\Omega_{\bar{t}(v)}$.
\qed

\smallskip
\noindent
\textbf{Proof of Theorem \ref{SCP3}}. 
Note that Lemmas \ref{scpboundary} and \ref{scpboundary2} mainly rely on the availability of the WCP and the Hopf maximum principle for solution $v$. 
In the case $p>2$ and $\lambda \geq 0$, we know that the WCP holds by Theorem \ref{WCP2} and the boundary estimate is imposed by the assumption \eqref{eq:SCP3} of the theorem. Therefore, following the same arguments as in the proofs of Lemmas \ref{scpboundary}, \ref{scpboundary2}, and Theorems \ref{SCP}, \ref{SCP2}, we obtain the desired result.
\qed

\section{Discussion}\label{sec:Discuss}

Our problem \eqref{eq:P} is somewhat related to
the degenerate diffusion problem (for $2 < p < \infty$)
with an inhomogeneous logistic reaction function treated in
{\sc S.\ Takeuchi} \cite[Sect.~3, p.~e1015]{takeuchi}:
\begin{equation}
\label{eq:EP}
\tag{$\mathrm{EP}$}
\left\{
\begin{alignedat}{2}
\partial_t u - \Delta_p u
& {} = \lambda\, |u|^{p-2} u (a(x) - u)  + f(x) \,,\quad
&& (x,t)\in \Omega_T \,,
\\[0.4em]
u(x,0)
& {} = u_0(x) \,,
&& x\in \Omega \,,
\\[0.4em]
u(x,t)
& {} = 0 \,,
&& (x,t)\in \partial\Omega_T \,.
\end{alignedat}
\quad
\right.
\end{equation}
Here,
\begin{math}
\Delta_p u\stackrel{\textrm{def}}{=}
\mbox{div}(|\nabla_x u|^{p-2} \nabla_x u)
\end{math}
is the $p$-Laplacian with the spatial gradient $\nabla_x u$, $p > 2$,
$\lambda\in \mathbb{R}_+ = [0,\infty)$, and both
$a,f\in L^{\infty}(\Omega)$ are some non\-negative functions,
$a\not\equiv 0$ in $\Omega$.
The corresponding semilinear problem with $p=2$ has been widely studied
in the literature, but the quasilinear analogue with $p\neq 2$
is less known; cf.\ \cite{derlet, takeuchi}.

Our methods developed in the present work for $2 < p < \infty$
(and $L^{\infty}(\Omega_T)$\--solutions)
are aplicable also to problem \eqref{eq:EP},
owing to the fact that the logistic reaction function
\begin{equation*}
g(x,\,\cdot\,)\colon \mathbb{R}\to \mathbb{R}
s \;\longmapsto\; \lambda\, |s|^{p-2} s (a(x) - s)
\end{equation*}
satisfies the one\--sided Lipschitz condition on
the non\-negative half\--line $\mathbb{R}_+$, by
\begin{equation*}
\frac{\partial g}{\partial s} (x,s) =
\lambda \, p\, |s|^{p-2}\left( \genfrac{}{}{}1{p-1}{p}\, a(x) - s\right)
\leq L_p\equiv \mathrm{const} < \infty
\end{equation*}
for a.e.\ $x\in \Omega$ and for all $s\in \mathbb{R}_+$.
More detailed weak comparison results for problem \eqref{eq:EP}
can be found in
{\sc A.\ Derlet} and {\sc P.\ Tak\'a\v{c}} \cite{derlet}.
Further results on the existence, uniqueness, and long\--time
asymptotic behavior of weak solutions are established in
\cite{derlet, takeuchi}.

\bigskip
\noindent
\textbf{Acknowledgements.}
The first author was supported by the project LO1506 of
the Czech Ministry of Education, Youth and Sports.
Both authors would like to express their sincere thanks
to an anonymous referee for his detailed questions and comments
concerning our {\rm Theorem 2.1 (WCP)} and also
for drawing our attention to Ref.\ \cite{takeuchi}.

\end{document}